\documentclass[12pt,utf8,a4paper,english,natbib,upmath]{aclart}
\usepackage{aclraccourcis}

\title[Burnside rings and volume forms with logarithmic poles]{Burnside rings and volume forms with logarithmic poles}

\author{Antoine Chambert-Loir}
\address{%
Université Paris Cité, Institut de Mathématiques de Jussieu-Paris Rive Gauche, F-75013, Paris, France}
\email{antoine.chambert-loir@u-paris.fr}

\author{Maxim Kontsevich}
\address{%
Institut des Hautes \'Etudes Scientifiques, 35 route de Chartres, 91440 Bures-sur-Yvette, France}
\email{maxim@ihes.fr}

\author{Yuri Tschinkel}
\address{Courant Institute, NYU, 251 Mercer St.  New York, NY 10012, USA}
\address{Simons Foundation, 160 5th Av., New York, NY 10010, USA}
\email{tschinkel@cims.nyu.edu}

\ifx\pdfoutput\undefined\else
\usepackage[
pdftitle={Burnside rings and volume forms with logarithmic poles},
pdfauthor={A. Chambert-Loir, M. Kontsevich and Y. Tschinkel}]
{hyperref}
\fi

\begin{document}

\date{\today}
 
\begin{abstract}
We develop a theory of Burnside rings in the context
of birational equivalences of algebraic varieties
equipped with logarithmic volume forms.
We introduce a residue homomorphism and 
construct an additive invariant of birational morphisms.
We also define a specialization homomorphism.
\end{abstract}
 
\begin{altabstract}
Nous proposons une théorie d'anneaux de Burnside dans
le contexte de la géométrie birationnelle des variétés
algébriques munies d'une forme volume à pôles logarithmiques.
Nous introduisons un homomorphisme « résidu », construisons
un invariant additif des morphismes birationnels.
Nous définissons aussi un homomorphisme de spécialisation.
\end{altabstract}


\maketitle

\tableofcontents

\section{Introduction}

The study of birationality of algebraic varieties is a classical 
and well studied subject, with many open problems.
In some cases, it is interesting to study birational maps
preserving additional structure, for example
group actions, symplectic forms, or volume forms. 
Such a study is already implicit in many questions of birational geometry,
\eg,  in the notion of crepant resolution of singularities.

In this paper, we consider the case of varieties
endowed with volume forms with logarithmic poles
and develop  a formalism of Burnside rings 
along the lines of their counterpart introduced
by~\citet{KontsevichTschinkel-2019} to establish
the specialization of rationality, and 
its equivariant version by~\citet{KreschTschinkel-2022}.

Let $k$ be a field of characteristic zero.
For each integer~$n$,
we define 
\[ \Burn_n(k) \]
as the free abelian group
on birational equivalence classses of pairs $(X,\omega)$
consisting of an integral smooth  proper $k$-variety~$X$ of dimension~$n$
equipped with an $n$-form~$\omega$ with at most logarithmic poles.

The graded abelian group
\[ \Burn(k) = \bigoplus_{n\in\N} \Burn_n(k) \]
carries a ring structure, induced by taking products of varieties,
decomposed into irreducible components,
and equipped with the external product of the volume forms. 

In section~\ref{sec.residues},
we define morphisms of abelian groups
\[ \partial \colon \Burn_n(k) \to \Burn_{n-1}(k) .\]
When $X$ is smooth and the polar divisor of~$\omega$ has strict
normal crossings, the image of the class $[X,\omega]$ is given
by the following formula. Let $(D_\alpha)_{\alpha\in\mathscr A}$
be the family of irreducible components of the polar divisor of~$\omega$.
For each subset~$A$ of~$\mathscr A$, the intersection $D_A=\bigcap_{\alpha\in A} D_\alpha$ is a union of integral smooth varieties of codimension~$\Card(A)$;
taking iterated residues, we may equip it 
with a volume form with logarithmic poles~$\omega_A$.
Then
\[ \partial ([X,\omega]) = \sum_{\emptyset\neq A \subset\mathscr A}
 (-1)^{\Card(A)-1} [D_A, \omega_A] \cdot \mathbf T^{\Card(A)-1}, \]
where 
\[ \mathbf T = [\P^1, \mathrm dt/t]. \]
In particular, the existence of the map~$\partial$ relies on the 
birational invariance of this expression,
see theorem~\ref{theo.residue-bir}.

This construction is reminiscent of the boundary map in polar homology
\citep{KhesinRosly-2003,KhesinRoslyThomas-2004,GorchinskiyRosly-2015}.
However, apart from the obvious difference that we only record birational types
of strata, rather than the strata themselves,
our formula takes into account strata of all codimensions, rather
than those of codimension one.

The map~$\partial$  is additive.
Furthermore, we prove in theorem~\ref{theo.partial-deriv} that
\[ \partial(a \cdot b) = \eps^n \cdot \partial (a)\cdot  b
+ a \cdot \partial(b) - \mathbf T \cdot  \partial(a) \cdot \partial(b), \]
when $a\in\Burn_m(k)$ and $b\in\Burn_n(k)$.
Here $\eps$ is the class of the point~$\Spec(k)$
equipped with the volume form equal to~$-1$.

In theorem~\ref{theo.dd=0}, we show that 
\[ \partial \circ \partial  = 0 . \]

Since $\eps^2=1$, the ring $\Burn(k)$ splits as a product
of two rings $\Burn^{\eps=1}(k)$ and $\Burn^{\eps=-1}(k)$, after inverting~$2$.
In \S\ref{sec.alg}, we explain the corresponding simplifications
of the formulas.

Inspired by the constructions of~\cite{LinShinderZimmermann-2020,LinShinder-2022,KreschTschinkel-2022a}, we define in~\S\ref{sec.biraut}
a homomorphism 
\[ \mathbf c \colon \Bir(X,\omega) \to \Burn_{n-1}(k), \]
from the group of birational automorphisms of the pair
$(X,\omega)$, where $X$ is an $n$-dimensional integral proper smooth variety
equipped with a logarithmic volume form~$\omega$.
As in the above references, our map~$\mathbf c$ is defined
at the groupoid level of birational maps preserving logarithmic volume forms.

When the birational isomorphism~$\phi \colon (X,\omega) \dashrightarrow (Y,\eta)$ is described
by a diagram 
\[ \begin{tikzcd}[column sep = small]
 & W \ar{dl}[swap]{p} \ar{dr}{q} \\
X \ar[dashrightarrow]{rr}{\phi} && Y \end{tikzcd} \]
of smooth proper integral $k$-varieties,
with birational morphisms~$p$ and~$q$, the two logarithmic
volume forms~$p^*\omega$ and~$q^*\eta$ on~$W$ are equal,
and the element~$\mathbf c(\phi)\in\Burn_{n-1}(k)$ is given by
\[ \mathbf c(\phi) = 
\sum_{E \in\operatorname{Exc}(q)} [E, p^*\omega_E] - 
\sum_{D \in\operatorname{Exc}(p)} [D, q^*\eta_D] \]
where $\operatorname{Exc}(p)$ is the set of irreducible components
of the exceptional divisor of~$p$, and where, for each such component~$D$,
the logarithmic volume form~$p^*\omega_D$ on~$D$
is obtained by taking the residue of~$p^*\omega$ along~$D$
(and similarly for~$q$).

Finally, consider a discrete valuation ring
with residue field~$k$ and field of fractions~$K$
and let $t$ be a uniformizing element.
In this context, we define a specialization map
\[ \rho_t \colon \Burn_n(K)\to \Burn_n(k). \]
The image of the class $[X,\omega]$  involves the combinatorics of a 
good model $(\mathscr X,\omega)$ over the valuation ring, 
and a certain subcomplex of the Clemens complex of the special fiber.
In the particular case where $\mathscr X$ is smooth,
the polar divisor of~$\omega$ is a relative divisor with normal crossings,
and denoting by~$\omega_o$ the restriction of~$\omega$ 
to the special fiber~$\mathscr X_o$,
one has
\[ \rho_t ([X,\omega])  = [\mathscr X_o,\omega_o]. \]
Note that 
the existence of such a specialization map implies,
as in theorem~1 of~\cite{KontsevichTschinkel-2019},
or as in \citep{NicaiseShinder-2019},
that birational equivalence of varieties with
logarithmic volume forms  is preserved under “good specializations”.

In the geometric case, where the valuation is the local ring of a curve~$C$
at point~$o$, the construction of the specialization map
can be viewed as a restriction to the special fiber
of a normalization of a global residue map~$\partial$
that takes place on a proper model whose special fiber is
 a divisor with normal crossings.
The normalization procedure extracts a subcomplex
of the Clemens complex of the special fiber.
A similar situation appeared in the study of Tamagawa
measures on analytic manifolds \citep{Chambert-LoirTschinkel-2010}.

Related constructions
emerged from the work of~\citet{KontsevichSoibelman-2006}
inspired by mirror symmetry,
and its subsequent developments,
\eg, by \citet{MustataNicaise-2015a, NicaiseXu-2016a, BoucksomJonsson-2017, JonssonNicaise-2020}.

The formalism developed in our paper has already been applied to
settle long-standing questions in higher-dimensional birational
geometry. For instance, \citet{LoginovZhang-2024}
were able to prove that on the complex projective space~$\P^n$ endowed with
the standard torus invariant volume form with logarithmic poles,
the group of volume preserving maps of the
is not generated by pseudo-reflections, if $n\geq4$;
in particular, this group is not simple.  
We anticipate analogs of powerful specialization theorems of Voisin,
and others.

Our constructions use essentially only formal properties of the residue maps.
Consequently, one can envision analogous theories 
for logarithmic forms of smaller degree, Milnor $K$-theory,
or even for the cycle modules of \citet{Rost-1996}.
However, our first attempts in this direction have not yet been successful.

\subsection*{Acknowledgments}
We are grateful to Hsueh-Yung Lin and Evgeny Shinder for having pointed out a lapsus 
in~\S\ref{sec.biraut}.
The third author was partially supported by NSF grant 2000099.

\section{Logarithmic differential forms}

\subsection{Kähler differentials}
Let $k$ be a field of characteristic zero
and $K$ be a finitely generated extension of~$k$;
let $n$ be its transcendence degree.
The space of Kähler differentials $\Omega_{K/k}$
is the $K$-vector space generated
by symbols $\mathrm da$, for $a\in K$, subject to the relations:
\begin{enumerate}
\item For $a\in k$, one has $\mathrm da = 0$;
\item For $a,b\in K$, one has $\mathrm d(a+b)=\mathrm da+\mathrm db$
and $\mathrm d(ab)=a\mathrm db + b(\mathrm da)$.
\end{enumerate}

For any integer~$m\geq0$, we may consider its $m$th exterior
power $\Omega^{m}_{K/k}$, which is a $K$-vector space
of dimension~$\binom nm$; in particular, it vanishes if $m>n$,
$\Omega^1_{K/k}$ has dimension~$n$,
and $\Omega^n_{K/k}$ has dimension~one.
One has $\Omega^0_{K/k}=k$, canonically.

Elements of $\Omega^n_{K/k}$, for $n=\trdeg_k(K)$,
are also called \emph{volume forms}.

For $a\in K^\times$, we also write $\dlog a=\mathrm da/a\in\Omega_{K/k}$.

\subsection{Models}
Let $m$ be an integer and 
let $\omega\in\Omega^m_{K/k}$.
A \emph{model} of~$K$ is an integral $k$-scheme~$X$
together with a $k$-isomorphism $K\simeq k(X)$; we say
that this model is proper, resp.\ smooth, if $X$ is proper,
resp.\ smooth over~$k$. 
Given such a model,
 $\omega$ induces a meromorphic global section~$\omega_X$
of~$\Omega^m_{X/k}$.
The polar ideal of~$\omega_X$ 
is the subsheaf of~$\mathscr O_X$ whose local sections 
are the $a\in\mathscr O_X$
such that $a\omega_X$ is induced by a regular $m$-form. 
Let $D$ be the zero-locus of this ideal. Its complement~$U$
is the largest open subscheme of~$X$ 
such that $\omega_X$ is induced by a regular $m$-form on~$U$.
If $X$ is smooth, then $\Omega^m_{X/k}$ is locally free,
hence the scheme~$D$ is an effective divisor (Hartogs's principle);
we call it the \emph{polar divisor} of~$\omega$ on~$X$.

\subsection{Logarithmic forms}
By Hironaka's theorem on embedded resolution of singularities,
there exist smooth projective  models $(X,\omega_X)$ of $(K,\omega)$
such that the polar divisor~$D$ of~$\omega_X$ has normal crossings.

Following \cite[chap.~2, \S3]{Deligne-1970}, we then say that  $\omega_X$ 
has at most logarithmic poles, or that $\omega$
has at most logarithmic poles on~$X$, 
if both $\omega_X$ and $\mathrm d\omega_X$ have at most simple poles along~$D$. 

The following lemma implies that this condition 
is essentially independent of the choice of~$X$
such that the polar divisor of~$\omega_X$ has normal crossings.

\begin{lemm}\label{lemm.logarithmic-functoriality}
Let $g\colon X'\to X$ be a morphism of smooth $k$-varieties,
let $D$ be a divisor with normal crossings in~$X$ and
let $D'$ be a divisor with normal crossings in~$X'$
such that $D'=g^{-1}(D)$.
Let $\omega$ be a regular $m$-form on $X\setminus D$
and let $\omega'=g^*\omega$.

\begin{enumerate}
\item If $\omega$ has at most logarithmic poles along~$D$, 
then $\omega'$ has at most logarithmic poles along~$D'$.

\item The converse holds if $g$ is proper and surjective.
\end{enumerate}
\end{lemm}
\begin{proof}
The first assertion is \citep[chap.~II, prop.~3.2, (iv)]{Deligne-1970}.
Let us prove the second one. 

Consider the
generic point~$\eta$ of~$X$ and a point $\eta'\in X'\setminus D'$
which is algebraic over~$k(\eta)$.
The Zariski closure~$X'_1$ of~$\eta'$ 
is proper and generically finite over~$X$,
and $D'_1=D'\cap X'_1$ is a divisor.
There is a proper modification $h\colon X'_2\to X'_1$ 
such that $X'_2$ is smooth and $D'_2=h^{-1}(D'_1)$ has normal crossings.
By the first part, the form $h^*\omega'|_{X'_1}$ has at most
logarithmic poles along~$D'_2$. 
Replacing~$g$ by $g\circ h$, we may assume 
that $g$ is generically finite. 

Since the sheaf of forms with at most logarithmic poles along~$D$
is locally free and $X$ is smooth, we can 
delete from~$X$ a subset of codimension at least~2.
Thus, we may assume that $g$ is flat,
$D$ is smooth and irreducible,
and $g$ is étale outside of $D$. 
It suffices to argue étale locally at the generic point of~$D$.
By the local description of ramified morphisms,
there are étale local coordinates $(z_1,\dots,z_n)$
on~$X$ such that $D_\red=V(z_1)$, 
local coordinates $(z'_1,\dots,z'_n)$ on~$X'$
such that $g^*z_1=(z'_1)^e$, $g^*z_2=z'_2$, etc.,
where $e$ is the ramification index of~$g$ along~$D$.
Let $d$ be the order of the pole of~$\omega$ along~$D$;
write $\omega = \alpha/z_1^d + \beta \wedge \mathrm dz_1/z_1^d$,
where $\alpha,\beta$ are regular forms which do not involve~$\mathrm dz_1$.
Then 
\[ \omega'=g^*\omega=g^*\alpha/(z'_1)^{de} 
+ e\, g^*\beta \wedge \mathrm dz'_1/(z'_1)^{1+(d-1)e}. \]
Since $\omega'$ has at most logarithmic poles along~$D$,
we get $g^*\alpha$ is divisible by $(z'_1)^{de-1}$ and $g^*\beta$ is divisible by $(z'_1)^{(d-1)e}$.
Assume, by contradiction, that $d\geq 2$, so that $de-1\geq 1$ and $(d-1)e\geq 1$.
This implies that both $g^*\alpha$ and $g^*\beta$
are multiples of~$z'_1$, so that both $\alpha$ are $\beta$ are divisible by~$z_1$, contradicting the hypothesis
that $d$ was the order of the pole of~$\omega$ along~$D$.
Therefore, $d\leq 1$. This concludes the proof.
\end{proof}

\subsection{}
We say that an $m$-form $\omega\in\Omega^m_{K/k}$ is  \emph{logarithmic}
if for all proper smooth models~$X$ of~$K$
such that the polar divisor of~$\omega_X$ has normal crossings,
the meromorphic differential form $\omega_X$ 
has at most logarithmic poles.

By resolution of singularities \citep{Hironaka-1964},
two models are dominated by a third one, 
hence lemma~\ref{lemm.logarithmic-functoriality}
implies that it suffices that this condition 
is satisfied on some proper smooth model for which
the polar divisor of~$\omega_X$ has normal crossings.

Analogously, if $X$ is a reduced $k$-variety, then we 
say that a meromorphic $m$-form~$\omega$ on~$X$ is logarithmic “everywhere”
if for all proper birational models $(X',\omega')$ of~$(X,\omega)$
such that $X'$ is smooth and the polar divisor of~$\omega'$ has normal
crossings,
the meromorphic $m$-form~$\omega'$ on~$X'$ has at most logarithmic poles.
It suffices that this holds on one such model.

\section{Burnside rings for logarithmic forms}

\subsection{Burnside rings}\label{ssec.burn}
Let $k$ be a field of characteristic zero and $n$
an integer such that $n\geq0$.
\citet{KontsevichTschinkel-2019}
defined the Burnside group $\burn_n(k)$ as the free abelian group
on isomorphism classes of 
finitely generated extensions of~$k$
of transcendence degree~$n$.

Any integral $k$-variety~$X$ of dimension~$n$
has a class $[X]$ in $\burn_n(k)$.
This gives rise to alternative useful presentations of~$\burn_n(k)$,
for example involving only classes of integral projective smooth varieties.

The group
\[
\burn(k)=\bigoplus_{n\ge 0}  \burn_n(k)
\]
carries a natural commutative ring structure, 
with multiplication defined by taking products of (smooth projective) $k$-varieties:
\[
[X]\cdot [X'] = [X\times X'].
\]

\subsection{Definition of a Burnside group for volume forms}
Let $k$ be a field of characteristic zero and let $n$ be an integer
$\geq0$.
We define $\Burn_n(k)$ to be the free abelian group
on isomorphisms classes of pairs $(K,\omega)$,
where 
\begin{itemize}
\item
$K$ is a finitely generated extension of $k$ of transcendence degree~$n$ and 
\item
$\omega\in\Omega^n_{K/k}$ is a logarithmic volume form.
\end{itemize}
We write 
\[
[K,\omega] \in \Burn_n(k)
\]
for the class of a pair $(K,\omega)$.

\begin{rema}
This definition has obvious more geometric formulations.
For example, we can take for generators equivalence classes
of pairs $(X,\omega)$, where 
\begin{itemize}
\item
$X$ is a smooth integral $k$-scheme of dimension~$n$,
and 
\item
$\omega$ a regular volume form on~$X$ which is logarithmic 
“everywhere”,
\end{itemize}
modulo the smallest equivalence relation
that identifies $(X,\omega)$ and $(X',\omega')$
if there exists an open immersion $f\colon X'\to X$
such that $\omega'=f^*\omega$. 

Alternatively, we can assume that $X$ is proper, smooth and integral, 
the form $\omega$ is a logarithmic volume form on~$X$,
and consider the smallest equivalence relation
that identifies $(X,\omega)$ and $(X',\omega')$
if there exists a proper birational morphism $f\colon X'\to X$
such that $\omega'=f^*\omega$.
By the weak factorization theorem of~\citep{AbramovichKaruMatsukiEtAl-2002},
this equivalence relation is generated by such morphisms~$f$
which are blowing-ups along smooth centers in good
position with respect to the polar divisor of~$X$.

In both contexts, if $X$ is an $n$-dimensional $k$-variety
and $\omega$ is a meromorphic $n$-form on~$X$ 
which is logarithmic “everywhere”,
then we
define $[X,\omega]$ to be the sum, over all irreducible components~$Y$
of~$X$ which have dimension~$n$,
of the classes $[Y,\omega|_Y]$.
\end{rema}

\begin{exem}\label{exem.eps}
Finitely generated extensions of~$k$ of transcendence degree~$0$
are finite extensions of~$k$. Let $K$ be such an extension.
Since $k$ has characteristic zero, 
one has $\Omega^1_{K/k}=0$. However, $\Omega^0_{K/k}$,
which is its $0$th exterior power, is canonically isomorphic to~$K$.
Consequently, 
$ \Burn_0(k) $ is the free  abelian group on isomorphism
classes of pairs $(K,\lambda)$, where $K$ is a finite extension of~$k$
and $\lambda\in K$.

We will let $\mathbf 1=[\Spec(k),1]$
and $\eps = [\Spec(k),-1]$.
\end{exem}

\begin{exem}\label{exem.T}
Let $K=k(t)$. The differential form $\mathrm dt/t$ 
is a logarithmic volume form; indeed $X=\P^1_k$ is a model
of~$K$ and this form has poles of order~$1$ at~$0$ and~$\infty$,
and no other poles. We write $\mathbf T$ for the class of
$(k(t),\mathrm dt/t)$. 

Note that the $k$-isomorphism of $K$ that maps~$t$ to~$1/t$
maps $\mathrm dt/t$ to its opposite;  consequently,
we also have $\mathbf T=[k(t),-\mathrm dt/t]=\eps\cdot\mathbf T$.

In the context of birational
geometry in presence of logarithmic volume forms,
“rational varieties” would have class in~$\mathbf T^n$,
and similarly for stable birationality.
\end{exem}

\subsection{Multiplicative structure}\label{ss.commut}
We view the direct sum 
\[ \Burn(k)=\bigoplus_{n\in\N}\Burn_n(k) \]
as a graded abelian group.
It is endowed with a multiplication such that
\[ [X,\omega] \cdot [X',\omega'] =  [X\times X', \omega\wedge \omega'] \]
when $X,X'$ are proper, smooth and integral
and $\omega$, resp.~$\omega'$ are logarithmic volume forms on~$X$, resp.~$X'$,
and $Y$ ranges over the set of irreducible components of $X\times X'$. 

Let $s\colon X'\times X\to X\times X'$ be the isomorphism
exchanging the two factors. One has
\[ s^*(\omega \wedge \omega') = (-1)^{nn'} \omega'\wedge\omega,  \]
if $n=\dim(X)$, $n'=\dim(X')$, $\omega$ is a volume form on~$X$ 
and $\omega'$ is a  volume form on~$X'$.
Consequently,
\[ a \cdot b = \eps^{nn'} \cdot b \cdot a \]
for $a\in\Burn_n(k)$ and $b\in \Burn_{n'}(k)$.
In particular, classes in $\Burn_n(k)$, for even~$n$, are central in~$\Burn(k)$.

We remark that 
the element~$\mathbf T\in\Burn_1(k)$ is central as well.
Let indeed $a\in\Burn_n(k)$. If $n$ is even, then $a\cdot\mathbf T=\mathbf T\cdot a$.
Otherwise, we have $a\cdot \mathbf T=\eps \cdot \mathbf T\cdot a$,
but we have seen in example~\ref{exem.T}
that 
$\mathbf T=\eps\cdot\mathbf T$. 
As a consequence, $a\cdot \mathbf T=\mathbf T\cdot a$.

However, the ring~$\Burn(k)$ is \emph{not} commutative.
Indeed, consider curves $X$, $X'$ 
without automorphisms
and no nonconstant morphism between them.
Then the switch is the only isomorphism
from $X'\times X$ to~$X\times X'$. 
Take nonzero logarithmic $1$-forms $\omega,\omega'$ on~$X,X'$ respectively. 
The classes $[X\times X',\omega\wedge\omega']$
and $[X'\times X,\omega'\wedge\omega]$ are then distinct.

\subsection{Functoriality}
Let $k'$ be an extension of~$k$. Then there is a natural
ring homomorphism 
\[ \Burn(k)\to\Burn(k') \]
described as follows.
Let $(X,\omega)$ be an integral $k$-variety
of dimension~$n$ equiped with a logarithmic $q$-form. 
Let $X'=X\otimes_k k'$
be its base change to~$k'$, and let $\omega'$
be the volume form on~$X'$ deduced from~$\omega$
by base change. 
Then the class of $(X,\omega)$ maps to 
the sum of classes $(Y,\omega'|_Y)$, where $Y$ runs
the (finite) set of irreducible components of~$X'$.


\subsection{Relation with the classical Burnside group}
Forgetting the form~$\omega$ gives a ring morphism
\[ \pi\colon \Burn(k) \to \burn(k). \]
On the other hand, if $K$ is a finitely generated
extension of~$k$ of transcendence degree~$n$, 
we can endow it with the zero $n$-form.
The resulting map
\[ \varpi\colon \burn(k) \to \Burn(k) \]
identifies $\burn(k)$ with an ideal of $\Burn(k)$.
One has $\pi\circ\varpi=\id$.

\subsection{Variations on the theme}
The construction of~$\Burn(k)$ admits
several natural variants that are relevant in more specific contexts.
Some of them will be used in later sections.

\subsubsection{A relative ring}
Let $n$ be an integer.
For any $k$-scheme~$S$, we define $\Burn_n(S/k)$
as the free abelian group on triples $(X,\omega,u)$
where $X$ is an integral smooth $n$-dimensional $k$-scheme,
$\omega\in\Omega^n_{X/k}$ is a regular volume form
which is logarithmic “everywhere”,
and $u\colon X\to S$ is a morphism, 
modulo the smallest equivalence relation 
that identifies $(X,\omega,u)$ and $(X',\omega',u')$
if there exists an open immersion $f\colon X'\to X$
such that $\omega'=f^*\omega$ and $u'=u\circ f$.

Let $h\colon S\to T$ be a morphism of $k$-schemes.
It induces a morphism of abelian groups 
\[
h_*\colon \Burn_n(S/k)\to \Burn_n(T/k)
\]
such that $h_*([X,\omega,u])=[X,\omega,h\circ u]$
for any triple $(X,\omega,u)$ as above.

\subsubsection{Pluriforms}
One can replace volume forms with volume $r$-pluriforms,
that is, elements of~$(\Omega^n_{K/k})^{\otimes r}$,
for some given integer~$r$. 
The corresponding logarithmic condition requires
that the pluriform has poles of order at most~$r$ on 
an adequate model.
Note that when $r$ is even, the obtained ring is commmutative.

\subsubsection{Forms up to scalars}
In the construction, we may wish to identify
$(K,\omega)$ and $(K',\omega')$
if there exists $\lambda\in k^\times$, resp. $\lambda \in\{\pm1\}$,
and a $k$-isomorphism $f\colon K\to K'$ such that $f^*\omega'=\lambda\omega$.
These variants also give rise to a commutative ring.

\subsubsection{Group actions}\label{sss.group}
Let $G$ be a profinite group scheme over~$k$.
One can also consider pairs $(K,\omega)$,
where the field~$K$ is endowed with an action of~$G$
leaving the form~$\omega$ invariant.
The obtained ring will be denoted by~
$ \Burn^G(k)$.

\section{Residues}\label{sec.residues}

\subsection{Residue of a volume form}
Let $X$ be an equidimensional smooth $k$-variety of dimension~$n$.

Let $D$ be a smooth divisor on~$X$.
We denote by $\Omega^m_{X/k}(\log D)$
the sheaf of $m$-forms on~$X$ with logarithmic poles along~$D$, 
locally of the form $\eta \wedge \mathrm d\log f + \eta' $,
where $\eta$ and $\eta'$ are regular and $f$ is a local equation of~$D$.
The residue map is the homomorphism  of $\mathscr O_X$-modules
\[
\rho_D\colon \Omega^{m}_{X/k}(\log D)\to \Omega^{m-1}_{D/k},
\]
characterized by the relation
\[
\rho_D(\eta \wedge \mathrm d \log f + \eta' )
=\eta|_D
\] 
for every local sections $\eta\in\Omega^{m-1}_{X/k}$
and $\eta'\in\Omega^m_{X/k}$, 
and any local generator~$f$ of the ideal of~$D$.

If $\omega$ is a logarithmic $m$-form on~$X$,
there is an open subset~$U$ of~$X$ such that $U\cap D\neq\emptyset$
and such that $\omega|_U$ belongs to $\Omega^m_{X/k}(\log D)$.
Its residue $\rho_D(\omega|_U)$ is then
a meromorphic section of~$\Omega^{m-1}_{D/k}$.

\begin{lemm}\label{lemm.residue-log}
Let $\omega$ be a logarithmic differential form of degree~$m$ on~$X$.
Then $\rho_D(\omega)$ is a logarithmic $(m-1)$-form on~$D$.
\end{lemm}
\begin{proof}
We may assume that the sum of~$D$ and of the polar divisor of~$\omega$
has strict normal crossings. The assertion 
is then evident in  local  coordinates.
\end{proof}

\subsection{Blowing-ups and normal bundles}
 
Let $Y$ be a smooth closed subscheme of~$X$.
The blow-up $\Bl_Y(X)$ of~$X$ along~$Y$ is a smooth $k$-variety.
The blowing-up morphism $b_Y\colon \Bl_Y(X)\to X$ is an isomorphism
over the complement of~$Y$.
If $Y$ is nowhere dense and nonempty,
then $E_Y=b_Y^{-1}(Y)$ is a smooth divisor in~$\Bl_Y(X)$. 

In general, $E_Y=b_Y^{-1}(Y)$ identifies, as a $Y$-scheme, with
the projectivization of the normal bundle $\mathscr N_Y(X)$
of~$Y$ in~$X$.  We write $j_Y\colon E_Y\to X$ for the natural
projection of~$E_Y$ to~$X$.

Let $W$ be a closed smooth subscheme of~$X$. Assume that $W$ and~$Y$ 
are transversal. Then the Zariski closure of $b_Y^{-1}(W\setminus (Y\cap W))$
is called the strict transform of~$W$ in~$\Bl_Y(X)$.
It identifies with $\Bl_{Y\cap W}(W)$.

Let $\omega$ be a logarithmic $m$-form on~$X$.
Then the form $b_Y^*\omega$ on~$\Bl_Y(X)$ is logarithmic;
assuming that $Y$ is nonempty and nowhere dense, we can consider
the residue $\rho_Y(\omega)$
of~$b_Y^*\omega$ along~$E_Y$. 
It is a logarithmic $(n-1)$-form on $E_Y$.

\begin{defi}\label{defi.rho}
Let $X$ be an irreducible proper smooth $k$-variety of dimension $n$ and
$\omega\in\Omega^n_{X/k}$ a  logarithmic volume form
whose polar divisor~$D$ has strict normal crossings.
Let $(D_\alpha)_{\alpha\in\mathscr A}$ be the family of its
irreducible components; for $A\subset\mathscr A$,
we let $D_A=\bigcap_{\alpha\in A}D_\alpha$.
We define an element $\rho(X,\omega)$ in $\Burn_{n-1}(X/k)$
by the formula:
\[ \rho(X,\omega) = \sum_{\emptyset\neq A\subset \mathscr A}
 (-1)^{\Card(A)-1}  [E_{D_A}, \rho_{D_A}(\omega), j_{D_A}] .\]
\end{defi}
In this formula and all similar ones below, it is always
implicit that the terms where $D_A=\emptyset$ are omitted.
We will also set
\[ \rho_{D_A}(X,\omega) = [E_{D_A}, \rho_{D_A}(\omega), j_{D_A}].\]

\subsection{Iterated residues}\label{ss.iter}
We retain the notation of definition~\ref{defi.rho}

Fix 
a logarithmic volume form~$\omega$ on~$X$
and 
a nonempty subset~$A$ of~$\mathscr A$
such that $D_A\neq\emptyset$.
It will be useful to compute  inductively
the logarithmic volume form $\rho_{D_A}(\omega)$
that appears in definition~\ref{defi.rho}.

Let $b_A\colon \tilde X\to X$ be the blowing-up of~$X$
along~$D_A$ and let $E$ be its exceptional divisor.

When $A=\{\alpha\}$ has a single element, 
$D_A$ is the divisor~$D_\alpha$,
the blowing-up morphism~$b_A$ is an isomorphism,
and the exceptional divisor identifies with~$D_A$.
Then 
\[ \rho(X,\omega)= \rho_{D_\alpha}(X,\omega)
 = [D_\alpha,\rho_{D_\alpha}(\omega),j_{D_\alpha}], \]
where  $j_{D_\alpha}$ is the immersion of~$D_\alpha$ into~$X$.

This construction can be pursued in higher codimension,
using iterated residues.
Fix a total order on~$\mathscr A$.
There is a unique, strictly increasing sequence
$(\alpha_1,\dots,\alpha_m)$ in~$\mathscr A$ such that 
$A=\{\alpha_1,\dots,\alpha_m\}$.
Given the chosen order on~$\mathscr A$,
we may apply the iterated residues construction 
and obtain a logarithmic form of degree~$n-m$
\[ \rho_{D_A} (\omega) 
= \rho_{D_{\alpha_1}} \circ \dots \circ \rho_{D_{\alpha_m}}(\omega). \]
On a nonempty open subset~$U$ of~$X$ that meets~$D_A$, we may write
\[ \omega=\eta \wedge \dlog(f_{\alpha_1})\wedge \dots
\dlog (f_{\alpha_m}), \]
for a regular form~$\eta$, 
and then one has $\rho_{D_A}(\omega)=\eta|_{U\cap D_A}$.

Denote by~$b_A$ the blowing-up of~$X$ along~$D_A$
and by~$E_A$ its exceptional divisor;
recall that $E_A$ identifies with the projectivized normal
bundle~$\mathscr N_{D_A}(X)$ of~$D_A$ in~$X$.
Using local equations for the divisors~$D_\alpha$, for $\alpha\in A$,
we trivialize $\mathscr N_{D_A}(X)$ on a dense open subscheme of~$D_A$;
this gives a birational isomorphism of~$E_A$ with~$D_A \times \P^{m-1}$
(with $m=\Card(A)$), and a local computation gives the formula
\[ [E_A, \rho_{E_A}(\omega) ] = 
[D_A,\rho_{D_A}(\omega)]\cdot \mathbf T^{m-1} \]
 in $\Burn_{n-1}(D_A/k)$.
In particular,
\[ \rho_{D_A}(X,\omega) = [D_A,\rho_{D_A}(\omega)] \cdot\mathbf T^{m-1}
\]
in $\Burn_{n-1}(X/k)$.

When $m\geq 2$, 
the definition of~$\rho_{D_A}$ actually depends 
on the chosen order of~$\mathscr A$, but only 
up to a sign, so that  the class $[D_A,\rho_{D_A}(\omega)]$
is well defined up to multiplication
by the class~$\eps\in\Burn_0(k)$. 
On the other hand,
it is multiplied by $\mathbf T^{m-1}$
and we have observed that $\eps \cdot \mathbf T=\mathbf T$.

\begin{prop}\label{prop.good-blow-up}
Let $(X,\omega)$, $D$, and $(D_\alpha)_{\alpha\in \mathscr A}$
be as in definition~\ref{defi.rho}.
Let $Y$ be a strict irreducible smooth subvariety of~$X$
such that $\codim_X(Y)\geq2$.
Let $\mathscr A_Y$ be the set of all $\alpha\in \mathscr A$
such that $Y\not\subset D_\alpha$;
we assume that $\sum_{\alpha\in \mathscr A_Y} D_\alpha$
meets~$Y$ transversally.

Let $g\colon X'\to X$ be the blowing-up of~$X$ along~$Y$
and let $\omega'=g^*\omega$; it is a logarithmic
form, its polar divisor
has strict normal crossings, and we have 
\[ g_* \rho (X',\omega') = \rho(X,\omega) \]
in $\Burn(X/k)$.
\end{prop}
\begin{proof}
Let $E=g^{-1}(Y)$ be the exceptional divisor;
for each $\alpha\in\mathscr A$, let $D'_\alpha$
be the strict transform of~$D_\alpha$.
The blow-up~$X'$
is smooth;
the  divisor $E+\sum_{\alpha\in\mathscr A}D'_\alpha$
has strict normal crossings and
contains the polar divisor of~$\omega'$.

Let $B$ be the set of all~$\beta\in\mathscr A$ such that $Y\subset D_\beta$,
so that $D_B$ is the minimal stratum containing~$Y$.

We split the discussion into two cases.

\begin{enumerate}
\item
\emph{Assume that $\dim(Y)< \dim(D_B)$.} 
Since $g$ is ramified along~$E$, its Jacobian vanishes along~$E$.
Since $\omega$ has poles of order at most~one, 
the form $\omega'=g^*\omega$ is regular at the generic point of~$E$,
hence the polar divisor of~$\omega'$ does not contain~$E$.
Since $g$ is a local isomorphism around
the generic points of~$D_\alpha$, for $\alpha\in\mathscr A$,
we thus see that the polar divisor of~$\omega'$ 
is equal to $\sum_{\alpha\in{\mathscr A}}D'_\alpha$.
Consequently, we have to compare
\[ \sum_{\emptyset\neq A\subset \mathscr A} (-1)^{\Card(A)-1}
  \rho_{D'_A} (X',\omega') 
\]
with
\[ \sum_{\emptyset\neq A\subset \mathscr A} (-1)^{\Card(A)-1}
  \rho_{D_A} (X,\omega) .\]

For every nonempty subset~$A$ of~$\mathscr A$,
one has
\[ g_* \rho_{D'_A}(X',\omega')  = \rho_{D_A}(X,\omega) \]
for every nonempty subset~$A$ of~$\mathscr A$,
which implies the desired formula in this case.

\item
\emph{Assume that $\dim(Y)=\dim(D_B)$.}
In this case, $Y$ is an irreducible component of~$D_B$. 
Since $D_{\emptyset}=X$ and $Y\neq X$, we have $B\neq\emptyset$.
We have to compare the expression
\[ \sum_{\emptyset\neq A\subset \mathscr A} (-1)^{\Card(A)-1}
  \rho_{D'_A} (X',\omega') 
+ \sum_{A\subset\mathscr A} (-1)^{\Card(A)}  \rho_{E\cap D'_A}(X',\omega') \]
with
\[ \sum_{\emptyset\neq A\subset \mathscr A} (-1)^{\Card(A)-1}
  \rho_{D_A} (X,\omega) .\]
The argument takes place in a neighborhood of~$Y$, 
which allows to assume that $Y=D_B$.

Let $A$ be a nonempty subset of~$\mathscr A$.
One has $D'_A=\emptyset$ whenever $B\subset A$, and the corresponding
terms are absent from the second expression.
On the other hand, if $B\not\subset A$, 
the morphism~$g$  identifies~$D'_A$ with the blow-up
of~$D_A$ along~$D_{A}\cap Y=D_{A\cup B}$. 
In particular, $g$ induces a birational isomorphism from~$D'_A$ to~$D_A$,
so that $g_* \rho_{D'_A}(X',\omega')=\rho_{D_A}(X,\omega)$.
Moreover, $E\cap D'_A$ is the projectivized
normal bundle $\mathbf P\mathscr N_{D_{A\cup B}}(D_A)$,
and 
\[ g_* \rho_{E\cap D'_A} (X',\omega')= \rho_{D_{A\cup B}}(X,\omega). \]
Similarly, one has 
\[ g_* \rho_E(X',\omega') = \rho_{D_B}(X,\omega). \]
This gives a formula of the form 
\begin{align*}
 g_*\rho(X',\omega') & = \sum_{\substack{\emptyset \neq A \subset \mathscr A \\ B \not\subset A}}  (-1)^{\Card(A)-1} \rho_{D_A}(X,\omega)
+ \sum_{\substack{ A \subset \mathscr A \\ B \not\subset A}} (-1)^{\Card(A)} \rho_{D_{A\cup B}}(X,\omega)  \\
& = \sum_{\emptyset \neq A \subset\mathscr A} n'_A \rho_{D_A}(X,\omega), 
\end{align*}
where
\[
 n'_A = \begin{cases} (-1)^{\Card(A)-1} & \text{if $B\not\subset A$,} \\
 \sum_{\substack{ C \subset \mathscr A \\ B \not\subset C 
 \\ C \cup B = A }} (-1)^{\Card(C)} & \text{if $B\subset A$.} \end{cases} \]
It suffices to prove that $n'_A=(-1)^{\Card(A)-1}$ 
for any nonempty subset~$A$ of~$\mathscr A$.
This is obvious when $B\not\subset A$, so let us assume that $B\subset A$.
In the sum that defines~$n'_A$,  we write $C=(C\setminus B)\cup C'$,
where $C'=C\cap B$ is a subset of~$B$;
the condition $C\cup B=A$ means $C\setminus B=A\setminus B$;
the condition $B\not\subset C$ means $C'\neq B$.
Consequently, we have
\begin{align*}
 n'_A & = (-1)^{\Card(A\setminus B)}
\sum_{\substack{C' \subset B \\ C' \neq B}} (-1)^{\Card(C')} \\
& = (-1)^{\Card(A\setminus B)}
\left( \sum_{C' \subset B } (-1)^{\Card(C')} 
- (-1)^{\Card(B)}\right) \\
& = (-1)^{\Card(A\setminus B)}
\left( (1-1)^{\Card(B)} - (-1)^{\Card(B)}\right) \\
& = (-1)^{\Card(A)-1}, 
\end{align*}
since $\Card(B)\geq1$.
This concludes the proof of the proposition.
\qedhere
\end{enumerate}
\end{proof}

\begin{theo}\label{theo.residue-bir}
Let $(X,\omega)$ be as in definition~\ref{defi.rho}.
If $X$ is proper, then the image of $\rho(X,\omega)$
in~$\Burn_{n-1}(k)$ only depends on the class $[X,\omega]\in
\Burn_n(k)$.
It gives rise to a morphism of abelian groups
\[ \partial_n\colon \Burn_n(k)\to \Burn_{n-1}(k). \]
\end{theo}
\begin{proof}
By the definition of $\Burn_n(k)$ involving
pairs $(X,\omega)$ where $X$ is proper,
it suffices to consider two pairs $(X,\omega)$
and $(X',\omega')$ as in definition~\ref{defi.rho}
which are related by a proper birational morphism
$g\colon X'\to X$ such that $g^*\omega=\omega'$.
By the weak factorization theorem of~\cite{AbramovichKaruMatsukiEtAl-2002},
in order to prove the theorem,
we may assume that $g$ is a blowing-up
of~$X$ along a smooth subvariety which is transversal
to the polar divisor of~$\omega$. In
this case, proposition~\ref{prop.good-blow-up} 
asserts that $g_*\rho(X',\omega')=\rho(X,\omega)$
in $\Burn(X/k)$. In particular,
the images in~$\Burn(k)$
of $\rho(X',\omega')$ and $\rho(X,\omega)$
are equal.
\end{proof}

\begin{rema}
The assumption that the form is logarithmic is essential in this construction.
When extended to general meromophic $n$-forms,
the map~$\rho$ given in definition~\ref{defi.rho} 
is not a birational invariant.
This is already apparent in the case of surfaces.
\end{rema}

\begin{exem}\label{exem.residue-torus}
The meromorphic differential form $\mathrm dt/t$ on~$\P^1_k$
has residues~$1$ and~$-1$ at~$0$ and~$\infty$ respectively.
By construction, we thus have 
\[ \partial_1(\mathbf T) = [\Spec(k), 1] + [\Spec(k),-1]
 = \mathbf 1 + \eps . \]

Let $n$ be an integer such that $n\geq 2$ 
and let us compute $\partial_n(\mathbf T^n)$.
We view $\mathbf T^n$ as the class of~$\P^n$, with homogeneous
coordinates $[1:x_1:\dots:x_n]$, and with the toric differential
form
\[
\omega_n=(\mathrm dx_1/x_1)\wedge \dots (\mathrm dx_n/x_n).
\]
Its divisor is the sum of the toric hyperplanes $D_0,\dots,D_n$.
Each of these hyperplanes identifies with $\P^{n-1}$, 
and $\rho_{D_j}(\omega_n)$ is $(-1)^{n-j}\omega_{n-1}$.
Let $\mathscr A=\{0,\dots,n\}$.
If $A=\mathscr A$, then $D_A=\emptyset$.
Otherwise, we see by induction
that  $D_A$ is isomorphic to~$\P^{n-\Card(A)}$
and $\rho_{D_A}(\omega_n)$ identifies with $\pm \omega_{n-\Card(A)}$,
so that 
\[ [D_A,\rho_{D_A}(\omega_n)]\cdot \mathbf T^{\Card(A)-1} = 
[\gm^{n-1},\pm \omega_{n-1}] = \mathbf T^{n-1}, \]
since $n-1\geq 1$.
Then,
\begin{align*}
\partial_n(\mathbf T^n) 
& = \sum_{\emptyset \neq A\subset \mathscr A}
 (-1)^{\Card(A)-1} [D_A, \rho_{D_A}(\omega_n)] \cdot \mathbf T^{\Card(A)-1} \\
& = \sum_{\emptyset \neq A\subsetneq \mathscr A} (-1)^{\Card(A)-1} \mathbf T^{n-1}.
\end{align*}
Now,
\[ \sum_{\emptyset \neq A\subsetneq\mathscr A} (-1)^{\Card(A)-1}
= 1 - (1-1)^{n+1}+ (-1)^{n+1} = 
\begin{cases} 
 2 & \text{if $n$ is odd;} \\
0 & \text{if $n$ is even}. 
\end{cases} \]
We get 
$ \partial_n(\mathbf T^n) = 2 \mathbf T^{n-1}$ 
if $n$ is odd and $\partial_n(\mathbf T^n)=0$ if $n$ is even.
(Recall that $n\geq 2$.)
Since $\mathbf T=\eps\cdot\mathbf T$, the following formula unifies 
the various cases: for $n\geq 1$, we have
\[ \partial_n(\mathbf T^n) = (1+(-1)^{n-1}\eps)\cdot \mathbf T^{n-1}. \]
\end{exem}

\begin{theo}\label{theo.partial-deriv}
Let $a\in\Burn_m(k)$ and $b\in\Burn_n(k)$;
we have
\[ \partial_{m+n}(a\cdot b)
= \eps^n\cdot   \partial_m(a) \cdot b + a \cdot \partial_n(b) - \mathbf T \cdot \partial_m(a)\cdot \partial_n(b) \]
in $\Burn_{m+n-1}(k)$.
\end{theo}
\begin{proof}
It suffices to treat the case where $a$ and $b$ are classes 
of proper integral smooth varieties $(X,\omega)$,
$(Y,\eta)$, endowed with meromorphic volume forms
whose polar divisors have strict normal crossings and no multiplicities.
Let $(D_\alpha)_{\alpha\in\mathscr A}$ be the irreducible
components of the polar divisor of~$\omega$,
let $(E_\beta)_{\beta\in\mathscr B}$ be the irreducible
components of the polar divisor of~$\eta$.
Then $[X,\omega]\cdot [Y,\eta]$
is the class of $[X\times Y, \omega\wedge\eta]$;
the polar divisor of~$\omega\wedge\eta$ is equal to
\[ \sum_{\alpha\in\mathscr A} D_\alpha \times Y + \sum_{\beta\in \mathscr B} X \times E_\beta. \]
We fix a total order on the disjoint union of~$\mathscr A$
and~$\mathscr B$ such that the elements of~$\mathscr A$ 
are smaller than those of~$\mathscr B$.
For any subsets $A,B$ of~$\mathscr A$ and~$\mathscr B$,
observe that we have
\[ \rho_{D_{A\cup B}} (\omega\wedge\eta) 
 = \pm \rho_{D_A}(\omega) \wedge \rho_{E_B}(\eta), \]
where $\rho_{D_A}$ has to be understood as
the identity when $A$ is empty, and similarly for~$\rho_{E_B}$.
The sign is~$1$ when $A=\emptyset$;
when $B=\emptyset$, it is equal to~$(-1)^{\Card(A)n}$;
we won't need to use its explicit value in the other cases.
Then we can write
$ \partial ([X,\omega]\cdot [Y,\eta]) $ as
\[
 \sum_{\substack{A\subset \mathscr A \\ B\subset \mathscr B \\
 A\cup B\neq\emptyset}}
(-1)^{\Card(A)+\Card(B)-1}
[ D_A\times E_B, \pm \rho_{A}(\omega)\wedge\rho_{E_B}(\eta)]
 \cdot \mathbf T^{\Card(A\cup B)-1} \]
and we split it into the sum of three terms,
according to which $B=\emptyset$, or $A=\emptyset$, 
or none of them is empty.
The first two terms are respectively equal to 
\[ \sum_{\emptyset\neq A\subset\mathscr A}
(-1)^{\Card(A)-1} [D_A\times Y,(-1)^{n\Card(A)}\rho_{D_A}(\omega)\wedge\eta]
\cdot \mathbf T^{\Card(A)-1}
= \partial([X,(-1)^n \omega])\cdot [Y,\eta] \]
and
\[ \sum_{\emptyset\neq B\subset\mathscr B}
(-1)^{\Card(B)-1} [X\times E_B, \omega\wedge \rho_{E_B}(\eta)]
\cdot \mathbf T^{\Card(B)-1}
= [X,\omega]\cdot \partial([Y,\eta]), \]
since $\mathbf T$ belongs to the center of~$\Burn(k)$.
As for the third one, 
we obtain
\[ - \sum_{\emptyset \neq B\subset \mathscr B} 
 (-1)^{\Card(B)-1} \sum_{\emptyset \neq A\subset \mathscr A}
 (-1)^{\Card(A)-1} [D_A, \rho_{D_A}(\omega)] \cdot [E_B,\rho_{E_B}(\eta)]
 \cdot \mathbf T^{\Card(A)+\Card(B)-2}  \]
which equals
\[ - \partial ([X,\omega]) \cdot \partial([Y,\eta])\cdot \mathbf T. \]
Finally, we get
\begin{align*}
 \partial_{m+n}(a\cdot b)  & = 
 \partial_{m+n}([X,\omega]\cdot [Y,\eta])  \\
& = \partial_m([X,(-1)^n\omega) \cdot [Y,\eta]
+ [X,\omega] \cdot \partial_n([Y,\eta]) \\
& \qquad{}  - \mathbf T \cdot \partial_m([X,\omega]) \cdot \partial_n([Y,\eta]) \\
& = \eps^n \cdot \partial_m(a) \cdot b + a \cdot \partial_n(b) 
  - \mathbf T \cdot \partial_m(a) \cdot \partial_n(b), 
\end{align*}
as claimed. 
\end{proof}
Using the computation of example~\ref{exem.residue-torus},
we obtain the following corollary.
\begin{coro} 
For any $a\in\Burn_m(k)$ and any positive integer~$n$, we have
\[ \partial_{m+n}(a\cdot \mathbf T^n) 
= \begin{cases}
\partial_m(a) \cdot \mathbf T^n  & \text{if $n$ is even;} \\ 
-\partial_m(a) \cdot \mathbf T^n +  a \cdot\partial_n(\mathbf T^{n}) 
& \text{if $n$ is odd.} \end{cases}
\]
\end{coro}

\begin{rema}
For the variant of $\Burn(k)$ where we considering forms up to sign,
the formula of theorem~\ref{theo.partial-deriv}
simplifies to 
\[ \partial_{m+n}(a\cdot b)
=  \partial_m(a) \cdot b + a \cdot \partial_n(b) - \mathbf T \cdot \partial_m(a)\cdot \partial_n(b). \]
\end{rema}

\section{A complex of Burnside rings}

\begin{theo}\label{theo.dd=0}
For any integer $n\geq 2$, we have
\[ \partial_{n-1} \circ \partial_n = 0. \]
\end{theo}
In other words,
the residue morphisms of Burnside groups give rise to a complex
\[ \dots \to \Burn_n(k) \to \Burn_{n-1}(k) \to \dots \to \Burn_1(k) \to\Burn_0(k) \]

\begin{proof}
It suffices to prove the following result:
Let $(X,\omega)$ be an integral proper smooth variety 
of dimension~$n$
equipped with a meromorphic volume form~$\omega$
whose polar divisor has strict normal crossings and 
no multiplicities;
then $\partial_{n-1}(\partial_n ([X,\omega]))=0$.

Let $(D_\alpha)_{\alpha\in\mathscr A}$ be the family
of irreducible components of the polar divisor of~$\omega$ in~$X$.
By definition, one has
\[ \partial_n ([X,\omega]) = 
\sum_{\emptyset\neq A\subset \mathscr A}
(-1)^{\Card(A)-1}  \rho_{D_A}(X,\omega). \]
Fix a total order on~$\mathscr A$.
Let $(\alpha_1,\dots,\alpha_m)$ be a strictly increasing sequence
in~$\mathscr A$ and let $A=\{\alpha_1,\dots,\alpha_m\}$.
We have seen in~\S\ref{ss.iter}
that $\rho_{D_A}(X,\omega)$ can be defined 
via iterated residue maps: 
\[ \rho_{D_A}([X,\omega])=
[D_A, \rho_{D_{\alpha_1}}\circ \dots \circ \rho_{D_{\alpha_m}} (\omega)]
\cdot \mathbf T^{\Card(A)-1}
= [D_A,\omega_A]\cdot \mathbf T^{\Card(A)-1} 
 \]
where we wrote $\omega_A$ for
the composition 
$\rho_{D_{\alpha_1}}\circ \dots \circ \rho_{D_{\alpha_m}}(\omega)$.
When $\Card(A)$ is odd, we have 
\[ \partial (\rho_{D_A}([X,\omega])) = \partial([D_A,\omega_A])
\cdot \mathbf T^{\Card(A)-1}, \]
while when $\Card(A)$ is even, we have 
\[ \partial (\rho_{D_A}([X,\omega])) 
= - \partial([D_A,\omega_A])\cdot \mathbf T^{\abs A-1}
+ [D_A,\omega_A] \cdot \partial(\mathbf T^{\Card(A)-1}) . \]
Consequently, we have
\[ \partial \circ\partial ([X,\omega])
= \sum_{\emptyset\neq A\subset \mathscr A}
 \partial ([D_A,\omega_A])\cdot \mathbf T^{\Card(A)-1}
- \sum_{\substack{\emptyset\neq A\subset \mathscr A \\ \text{$\Card(A) $ even}}}
[D_A,\omega_A]\cdot \partial(\mathbf T^{\Card(A)-1}). 
\]

The polar divisor of the form
$\omega_A $ on~$D_A$
is equal to $\sum_{\beta\not\in A} D_\beta \cap D_A$, 
so that, by definition (and computation of~$\partial$ via iterated residues), 
\[ \partial([D_A,\omega_A])
= \sum_{\emptyset \neq B \subset \complement A}
(-1)^{\Card(B)-1} [D_{AB},\omega_{A\cup B}] \cdot \mathbf T^{\Card(B)-1}.\]
Also, when $A$ is nonempty and of even cardinality, 
$\partial(\mathbf T^{\Card(A)-1})=2\mathbf T^{\Card(A)-2}$.
When we put these two formulas into the antepenultimate one
and collect the various terms, we obtain
\[ \partial \circ\partial ([X,\omega])
= \sum_{\substack{ C\subset \mathscr A \\ 2\leq \Card(C)}} n_C [D_C,\omega_C]  \cdot \mathbf T^{\Card(C)-2},
\]
where
\[ n_C = - \sum_{\substack{\emptyset \neq A, B  \\
 A \cup B = C, A \cap B = \emptyset}} (-1)^{\Card(B)}
 -  2 \delta_{\text{$\Card(C)$ is even}}. \]
In the first sum, the terms $A=\emptyset $ or $B=\emptyset$ are omitted,
while  if we put them in, we obtain
\[ \sum_{\substack{A \cup B=C \\ A \cap B=\emptyset}}
(-1)^{\Card(B)} = \sum_{b=0}^{\Card(C)} \binom {\Card(C)}b (-1)^b
= (1-1)^{\Card(C)}=0 \]
since $\Card(C)\geq 1$.
Consequently, 
\[ n_C = 1 + (-1)^{\Card(C)} - 2 \delta_{\text{$\Card(C)$  is even}}
= 0. \]
This concludes the proof.
\end{proof}

\section{Algebraic structure of \(\Burn(k)\) after localization at~\(2\)}
\label{sec.alg}

In this section, we study the algebraic structure of
the Burnside ring $\Burn(k)$, endowed with its elements~$\eps$, $\mathbf T$
and the operator~$\partial$. 

\subsection{}
By construction, $\Burn(k)=\bigoplus_{n\geq 0}\Burn_n(k)$
is an associative unital $\Z_{\geq0}$-graded ring, $\eps\in\Burn_0(k)$,
$\mathbf T\in\Burn_1(k)$ and $\partial$ is a homogeneous
additive map of degree~$-1$. They satisfy the following relations,
for homogeneous elements $a,b\in\Burn(k)$:
\begin{align*}
\tag{1}
 & b\cdot a = \eps^{\abs a\abs b} \cdot a\cdot b 
	&& \text{(\S\ref{ss.commut});} \\
\tag{2}
 & \eps^2 = 1 && \text{(example~\ref{exem.eps});}  \\
\tag{3}
 & \mathbf T = \eps\cdot \mathbf T && \text{(example~\ref{exem.T});} \\
\tag{4}
 & \partial(\mathbf T) = 1+\eps && \text{(example~\ref{exem.residue-torus});} \\
\tag{5}
 & \partial(a\cdot b) = \eps^{\abs b}\cdot\partial(a)\cdot b+ a\cdot \partial(b) -\mathbf T\cdot\partial(a)\cdot\partial(b) && \text{(theorem~\ref{theo.partial-deriv});} \\
\tag{6}
 & \partial(\partial(a)) = 0 &&  \text{(theorem~\ref{theo.dd=0}).} 
\end{align*}
By~(1), the element~$\eps$ is central, and by~(2), we may view~$\Burn(k)$
as an algebra over $\Z[\eps]/(\eps^2-1)$. After inverting~2, 
the algebra $\Burn(k)$ splits into two components
$\Burn^{\eps=1}(k)$ and $\Burn^{\eps=-1}(k)$, one over which $\eps=1$,
and the other over which $\eps=-1$.

\emph{In the rest of this section, we implicitly assume that $2$ is inverted,
without changing the notation.}

\subsection{Sector $\eps=-1$}\label{sect.sectorminus}
Here, we have $\mathbf T=-\mathbf T$, 
hence $\mathbf T=0$ since $2$ is invertible.
As a consequence, after replacing~$\partial$ with
$\partial'\colon a\mapsto (-1)^{1+\abs a}\partial(a)$,
one gets from~(5) the usual graded Leibniz rule 
\[ \partial'(a\cdot b)=\partial'(a)\cdot b+(-1)^{\abs a}a\cdot\partial'(b) \]
and therefore
$\Burn^{\eps=-1}(k)$ 
is a classical differential graded (super\nobreak-)commutative algebra,
similar to, \eg, the de Rham complex.

\subsection{Sector $\eps=1$}
The algebra $\Burn^{\eps=1}$ is now commutative (and not graded commutative).
This reflects the intuition in our constructions that they
speak about volume forms (as opposed to top-degree differential forms) 
for which we have commutativity (as reflected by the change of order
of integration in multiple integrals). 

\begin{lemm}
The map $F\colon a \mapsto  a- \mathbf T\cdot \partial(a)$
is a ring endomorphism of $\Burn^{\eps=1}(k)$, and $F^2=\id$.
Moreover, one has $F \circ \partial =\partial = - \partial \circ F$.
\end{lemm}
\begin{proof}
This map is additive. One has 
$ F(1) = 1 - \mathbf T\cdot \partial(1)=1$. Let us show multiplicativity.
Indeed, for $a,b\in \Burn^{\eps=1}(k)$, one has
\begin{align*}
F(a)\cdot F(b) & = (a-\mathbf T\cdot \partial(a))\cdot (b-\mathbf T\cdot\partial(b)) \\
& = a\cdot b - \mathbf T\cdot \partial(a)\cdot b-\mathbf T\cdot a\cdot\partial(b)
+\mathbf T^2 \cdot\partial(a)\cdot\partial(b)\\
& = a\cdot b - \mathbf T\cdot (\partial(a)\cdot b+ a\cdot\partial(b)
-\mathbf T \cdot\partial(a)\cdot\partial(b)) \\
& = a\cdot b - \mathbf T\cdot \partial(a\cdot b) & \text{(using~(5))} \\
& = F(a\cdot b). 
\end{align*}

Since $\partial^2=0$, one has 
\[ F (\partial(a))= \partial(a)-\mathbf T\cdot \partial(\partial(a))=\partial(a). \]
On the other hand,
\begin{align*}
\partial(F(a))& = \partial (a-\mathbf T\cdot \partial(a))  \\
& = \partial(a) - \partial(\mathbf T\cdot \partial(a)) \\
& = \partial(a) - \partial(\mathbf T) \cdot \partial(a)
- \mathbf T \cdot \partial(\partial(a))+\mathbf T\cdot \partial(\mathbf T)
\cdot \partial(\partial(a)) \\
& = -\partial(a)
\end{align*}
using that $\partial(\mathbf T)=2$ and $\partial^2=0$.

Consequently, for $a\in\Burn^{\eps=1}(k)$, we have
\[
F^2(a)  = F(a) - \mathbf T\cdot \partial (F(a)) 
 = a - \mathbf T\cdot \partial(a) + \mathbf T \cdot \partial (a) 
 = a \]
since $\partial\circ F=-\partial$.
\end{proof}

\subsection{}
To simplify the notation, write $\mathscr B=\Burn^{\eps=1}(k)$.
Since $F^2=\id$ and $2$ is invertible, the algebra $\mathscr B$
splits as a direct sum 
\[ \mathscr B= \mathscr B_+ \oplus \mathscr B_-, \]
such that $F$ acts as~$\id$ on $\mathscr B_+$
and as $-\id$ on~$\mathscr B_-$. Moreover, $\mathscr B_+$ is a subalgebra.

Since the operator~$\partial$ anticommutes with~$F$,
it induces maps
\[ \partial_{\pm}\colon \mathscr B_+ \to \mathscr B_-,
\qquad \partial_{\mp}\colon \mathscr B_- \to \mathscr B_+. \]

Note that 
\[ F(\mathbf T) = \mathbf T- \mathbf T \cdot \partial(\mathbf T)
= - \mathbf T , \]
so that $\mathbf T\in\mathscr B_-$.
Consequently, the multiplication by~$\mathbf T$ map induces
two maps
\[ t_{\pm}\colon \mathscr B_+\to\mathscr B_-, \qquad 
t_{\mp}\colon \mathscr B_-\to\mathscr B_+. \]

\begin{lemm}
The map~$\partial$ vanishes on~$\mathscr B_+$.
Equivalently, $\partial_{\pm}=0$. 

The maps $\frac12 \partial_{\mp}$ and~$t_{\pm}$ are inverses
the one of the other.
\end{lemm}
\begin{proof}
For $a\in\mathscr B_+$, one has $\partial(a)=-\partial(F(a))=-\partial(a)$,
since $\partial\circ F=-\partial$, hence $\partial(a)=0$.

On the other hand, for $a\in\mathscr B_+$, one has
\[ \partial (\mathbf T\cdot a)
= 2 \cdot a+\mathbf T\cdot \partial(a)- 2\mathbf T\cdot \partial(a)
= 2a - \mathbf T\cdot\partial(a)=a + F(a) = 2a \]
while for $a\in\mathscr B_-$, we have 
\[ \mathbf T\cdot \partial(a) = a-F(a) = 2a. \]
This concludes the proof of the lemma.
\end{proof}
In particular, we see that the cohomology of the differential $\partial$ vanishes in the sector $ \Burn^{\eps=1}(k) =\mathscr B$.

\subsection{}
It follows from the lemma that we have a ring isomorphism
\[ \mathscr B = \mathscr B_+[t]/(t^2-\mathbf T^2), \]
from which we see that all the algebraic structure of~$\mathscr B_+$
(namely $\delta$, $\mathbf T$, $F$) can be canonically reconstructed from a unital commutative associative $\Z_{\geq0}$-graded ring~$\mathscr B_+$ endowed with an element in degree $+2$ (namely, the element $\mathbf T^2$).
\begin{rema}
The situation clarifies even more if we invert the class~$\mathbf T$.
Then we can write $\partial(a)=(a-F(a))/\mathbf T$,
and all relations happen to follow from the fact that $F$ is an involution
such that $F(\mathbf T)=-\mathbf T$.
Indeed,
\[
\partial^2(a) = \frac{\partial(a)-F(\partial(a))}{\mathbf T}
= \frac1{\mathbf T} \left( \frac{a-\partial(a)}{\mathbf T}
- F( \frac{a-\partial(a)}{\mathbf T}) \right) = 0 \]
explains that $\partial^2=0$.
Moreover, for $a,b\in\mathscr B$, we  have
\begin{align*}
\partial(a\cdot b) & = \frac{a\cdot b - F(a\cdot b)}{\mathbf T}
 = \frac{a\cdot b - F(a)\cdot F(b)}{\mathbf T} \\
& = \frac{a-F(a)}{\mathbf T} \cdot b + a \cdot \frac{b-F(b)}{\mathbf T} 
 - \mathbf T \cdot \frac{a- F(a)}{\mathbf T} \cdot \frac{b-F(b)}{\mathbf T} \\
 & =\partial(a)\cdot b+a\cdot\partial(b)-{\mathbf T}\cdot\partial(a)\cdot\partial(b).
\end{align*}
\end{rema}

\section{Birational morphisms preserving volume forms}
\label{sec.biraut}

\subsection{}
Let $(X,\omega_X)$ be a smooth integral $k$-variety of dimension~$n$
equipped with a logarithmic volume form, 
and let $f\colon Y\to X$ be a proper birational morphism.

Let $E$ be an exceptional divisor in~$Y$,
that is, such that $\dim(f(E))<\dim(E)$. 
By lemma~\ref{lemm.logarithmic-functoriality},
and lemma~\ref{lemm.residue-log}, 
the residue $\rho_E(f^*\omega_X)$ along~$E$ 
of the meromorphic form~$f^*\omega$ is a logarithmic volume
form on~$E$.

We define $\mathbf c(f;X,\omega)$ to be the sum
of all such classes $[E,\rho_E(f^*\omega)]$
in the free abelian group $\Burn_{n-1}(k)$.

\begin{lemm}\label{lemm.c-additivity}
Let $g\colon Z\to Y$ be a proper birational morphism
of smooth integral varieties  of dimension~$n$.
Then $g\circ f$ is a proper birational morphism
and one has
\[ \mathbf c(g\circ f;X,\omega) = \mathbf c(g;Y,f^*\omega) + \mathbf c(f;X,\omega) \]
in $\Burn_{n-1}(k)$.
\end{lemm}
\begin{proof}
An integral divisor~$F$ in~$Z$ is exceptional for
$g\circ f$ if and only if one of the two 
mutually excluding situations happens:
\begin{itemize}
\item The divisor~$F$ is exceptional for~$g$;
\item Or $g(F)$ is a divisor in~$Y$ which is exceptional for~$f$.
\end{itemize}
Moreover, any divisor~$E$ in~$Y$ which is exceptional
for~$f$ appears once and only as  a divisor of the form~$g(F)$.
The contribution of~$F$ to~$\mathbf c(g\circ f;X,\omega)$
is given by the volume form~$\rho_F((g\circ f)^*\omega)$.
In the first case, we write
$\rho_F((g\circ f)^*\omega)=\rho_F(g^*(f^*\omega))$,
so that the contribution of~$F$ coincides with its contribution
to the term $\mathbf c(g;Y,f^*\omega)$.
In the second case, $g$ induces a birational isomorphism
from~$F$ to~$E=g(F)$;
writing $\rho_F((g\circ f)^*\omega)=g^*(\rho_F(f^*\omega))$,
we see that the contribution of~$F$ coincides with the
contribution of~$E$ to $\mathbf c(f;X,\omega)$.
This concludes the proof.
\end{proof}

\subsection{}
Let $(X,\omega_X)$ and $(Y,\omega_Y)$ be proper smooth $k$-varieties
equipped with logarithmic volume forms and 
let 
\[
\phi\colon (X,\omega_X) \dashrightarrow (Y,\omega_Y)
\]
be a birational map
preserving the volume forms.
By definition, this means that there exists a diagram
\[ \begin{tikzcd}[column sep = small]
 & W \ar{dl}[swap]{p} \ar{dr}{q} \\
X \ar[dashrightarrow]{rr}{\phi} && Y \end{tikzcd} \]
of integral $k$-varieties
such that $p$ and~$q$ are proper and birational,
and such that $p^*\omega=q^*\omega'$ on~$W$.
In this situation, we may assume that $W$ is smooth.

\begin{lemm}\label{lemm.c-independence}
With this notation, the element  
\[
\mathbf c(\phi)=\mathbf c(q)-\mathbf c(p) \in \Burn_{n-1}(k)
\]
only depends on the birational map~$\phi$,
and not on the choice of the triple $(W,p,q)$.
\end{lemm}

\begin{proof}
Consider two possible diagrams 
$X\xleftarrow p V\xrightarrow q Y$ 
and $X\xleftarrow r W\xrightarrow s Y$
describing~$\phi$. 
Considering for example a resolution
of singularities~$U$ of $V\times_X W$,
we can fit these two diagrams in a common commutative diagram
of the following form:
\[ \begin{tikzcd}[column sep = small]
 && U \ar{dl}[swap]{u} \ar{dr}{v} \\
 & V \ar{dl}[swap]{p} \arrow{drrr}[very near start]{q}
 	&& W \ar[crossing over]{dlll}[swap,very near start]{r\vphantom q} \ar{dr}{s} \\
X \ar[dashrightarrow]{rrrr}{\phi} &&&& 
Y 
\end{tikzcd} \]
The equalities $p^*\omega_X=q^*\omega_Y$
and $r^*\omega_X=s^*\omega_Y$ imply
that 
\[ (p\circ u)^*\omega_X = 
u^*p^*\omega_X = u^* q^*\omega_Y
=(q\circ u)^*\omega_Y = (s\circ v)^*\omega_Y. \]
By lemma~\ref{lemm.c-additivity}, we then have
\[
\mathbf c(p)-\mathbf c(q) 
 = \mathbf c(p\circ u)-\mathbf c(q\circ u)  
 =  \mathbf c(r\circ v)-\mathbf c(s\circ v)  
 = \mathbf c(r)-\mathbf c(s) . 
\]
This concludes the proof.
\end{proof}

\begin{theo}\label{theo.additivity}
If $\psi\colon (Y,\omega_Y)\dashrightarrow (Z,\omega_Z)$
is another birational map preserving volume forms,
then one has
\[ \mathbf c(\psi\circ\phi) = \mathbf c(\psi)+\mathbf c(\phi). \]
\end{theo}
\begin{proof}
Consider two diagrams 
$X\xleftarrow p V\xrightarrow q Y$ 
and $Y\xleftarrow r W\xrightarrow s Y$
describing~$\phi$ and~$\phi$. 
Considering for example a resolution
of singularities~$U$ of $V\times_Y W$,
we can fit these two diagrams in a common commutative diagram
of the following form:
\[ \begin{tikzcd}[column sep = small]
 && U \ar{dl}[swap]{u} \ar{dr}{v} \\
 & V \ar{dl}[swap]{p} \arrow{dr}{q}
 	&& W \ar{dl}[swap]{r\vphantom q} \ar{dr}{s} \\
X \ar[dashrightarrow]{rr}{\phi} && Y \ar[dashrightarrow]{rr}{\psi} && Z 
\end{tikzcd} \]
and the diagram $X\xleftarrow{p\circ u} U\xrightarrow{s\circ v}$
describes the birational map $\psi\circ \phi$.
Since $q\circ u=r\circ v$, we then have
\begin{align*}
\mathbf c(\psi\circ \phi) & = 
\mathbf c(p\circ u)-\mathbf c(s\circ v) \\
& = \mathbf c(p\circ u)-\mathbf c(q\circ u)
+ \mathbf c(r\circ v)- \mathbf c(s\circ v) \\
& = \mathbf c(p)-\mathbf c(q) + 
+ \mathbf c(r)- \mathbf c(s) \\
& = \mathbf c(\phi) + \mathbf c(\psi), 
\end{align*}
as was to be shown.
\end{proof}

\begin{coro}
Let $\Bir(X,\omega)$ be the set of birational automorphisms
of~$X$ preserving~$\omega$. The map~$\mathbf c$
induces a homomorphism of abelian groups
\[
\Bir(X,\omega)\to \Burn_{n-1}(k).
\]
Its kernel contains the group of automorphisms of~$X$
that preserve~$\omega$.
\end{coro}

\section{Specialization}

Let $K$ be the field of fractions of a discrete valuation ring~$R$
with residue field~$k$. 
Fix a uniformizer~$t\in R$.

In this context, 
\citet[\S3]{KontsevichTschinkel-2019} have defined
an additive \emph{specialization morphism}
\[\rho\colon  \burn_n(K)\to\burn_n(k), \]
relating the Burnside groups of~$K$ and~$k$ defined in \ref{ssec.burn}.
This map is \emph{not} multiplicative and 
they also modify the construction in \S5 to define a ring homomorphism
\[\rho_t\colon  \burn(K)\to\burn(k). \]
(Example~6.2 of~\citep{KreschTschinkel-2022} shows that
it actually depends on the choice of~$t$.) 

The goal of this section is to define a similar ring homomorphism
\[ \boldsymbol\rho_t \colon \Burn(K) \to \Burn(k) \]
for varieties with logarithmic volume forms.

The original construction of the invariant~$\rho_t$
is based on the introduction of an equivariant Burnside ring $\burn^\mu(k)$,
where $\mu$ is the profinite group scheme of roots of unity.
Indeed, in their paper, it is presented as a composition
\[ \rho_t\colon \burn(K) \to \burn^\mu(k) \to \burn(k), \]
where the second map is just forgetting the $\mu$-action.
However, the definitions of the ring $\burn^\mu(k)$
and of the first map $\hat\rho_t\colon \burn(K)\to \burn^\mu(k)$
are only sketched there; as a referee 
for the present paper pointed out, they are insufficient.
Here, we provide the necessary details.

\subsection{}\label{sec.kt}
Let us first recall the construction of the map~$\rho$
from~\citep{KontsevichTschinkel-2019}.
Let $\mathscr X$ be an integral proper scheme over~$R$,
of relative dimension~$n$,
whose special fiber~$\Delta$ is a divisor with strict normal crossings.

Let $(\Delta_\alpha)_{\alpha\in\mathscr A}$  
be the family of irreducible components of the special fiber~$\Delta$.
For every nonempty subset~$A$ of~$\mathscr A$,
let $\Delta_A$ be the intersection of all divisors~$\Delta_\alpha$;
let also $\Delta_A^\circ$
be the complement $\Delta_A\setminus \bigcup_{\alpha\not\in A}\Delta_\alpha$.

The first specialization morphism~$\rho$ of~\citep{KontsevichTschinkel-2019}
is defined by
\begin{equation}
 \rho( [\mathscr X_K]) 
= \sum_{\emptyset\neq A\subset\mathscr A} (-1)^{\Card(A)-1} [\Delta_A] 
\Lef^{\Card(A)-1},
\end{equation}
where $\Lef\in\burn(k)$ is the class of the affine line.
Before we proceed with  details about the construction of the
ring homomorphism~$\rho_t$, we need a refinement of this construction.

\begin{lemm}
Let $X$ be a proper smooth $K$-variety and $\mathscr X$ 
a proper $R$-model of~$X$ such that the inclusion $X\hookrightarrow \mathscr X$
is a toroidal embedding without self-intersection in the sense of
\cite[p.~57]{KempfKnudsenMumfordEtAl-1973}.
Let $(\Delta_\alpha)_{\alpha\in\mathscr A}$  
be the family of irreducible components of the special fiber~$\Delta$.
Then one has
\begin{equation}
 \rho( [\mathscr X_K]) 
= \sum_{\emptyset\neq A\subset\mathscr A} (-1)^{\Card(A)-1} [\Delta_A] 
\Lef^{\Card(A)-1}.
\end{equation}
\end{lemm}
\def\Tau{{\mathrm T}}
This lemma is implicit in the proof of the multiplicativity
of the refined invariant~$\rho_t$ in~\cite{KontsevichTschinkel-2019}.
A variant in the language of log-structures can be traced back to
\cite[Appendix]{NicaiseShinder-2019};
see also~\cite{BultotNicaise-2020}.
\begin{proof}
The scheme~$\mathscr X$ is not regular, 
but its singularities can be resolved by toroidal blow-ups.
Consider such a resolution $p\colon \mathscr Y\to\mathscr X$.
Let $(\mathrm E_\beta)_{\beta\in\mathscr B}$ be the familly of irreducible
components of the special fiber~$\mathrm E$ of~$\mathscr Y$.
By definition, one has 
\[ \rho(X) = \sum_{\emptyset\neq B\subseteq \mathscr B}
 (-1)^{\abs B-1}[\mathrm E_B] \Lef^{\abs B-1}. \]

The étale-local model for~$\mathscr X$ is a toric variety
$U_\sigma=\Spec(k[\sigma^\vee\cap \Z^n])$, where $\sigma$
is a simplicial cone in~$\R^n$.  Moreover, the $R$-structure
is given by a choice of a uniformizing element~$t\in R$
which corresponds to some element $m\in\sigma^\vee\cap\Z^n$.
In this toric chart, the resolution amounts to the combinatorial datum
of a subdivision of the cone~$\sigma$ into primitive simplicial subcones.
Let thus $\Tau$ be a fan with support~$\sigma$ 
whose cones are primitive and simplicial, 
let $U_\Tau$ be the associated toric variety, and
let $p\colon U_\Tau \to  U_\sigma$
be the canonical toric morphism. Note that $m\in\tau^\vee\cap\Z^n$
for every cone~$\tau$ of~$\Tau$.

The irreducible components~$\Delta_\alpha$ of the special fiber of~$\mathscr X$
that lie in the considered stratum correspond to the generators 
of the cone~$\sigma$, and strata~$\Delta_A$ correspond
to the subcones of~$\sigma$.
Similarly, the irreducible components~$\mathrm E_\beta$ 
that lie over this chart
correspond to the 1-dimensional fans  of~$\Tau$,
and the strata~$\mathrm E_B$ to subcones of the fan~$\Tau$.
For a subset $B\subset\mathscr B$, the map~$p$ induces
a morphism from the stratum~$\mathrm E_B$ 
to the stratum $\Delta_A$, 
where $A$ corresponds to the smallest subcone of~$\sigma$ containing~$\tau$,
with rational fibers. 
In particular, $[\mathrm E_B]=[\Delta_A] \Lef^{\abs A-\abs B}$.
Writing $B\mapsto A$ for this relation,
we thus can rewrite the definition of~$\rho(X)$ as
\[ \rho(X) = \sum_{\emptyset \neq A\subseteq \mathscr A}
 (-1)^{\abs A-1} [\Delta_A] \Lef^{\abs A-1}
 \sum_{B \mapsto A} (-1)^{\abs B-\abs A}. \]
The lemma follows by observing  
that for each~$A$, one has $\sum_{B\mapsto A} (-1)^{\abs B}=(-1)^{\abs A}$.
Indeed, $(-1)^{\abs B}$ is the Euler characteristic with compact support
of the open subcone corresponding to~$B$,  and the cone~$\sigma$
is their disjoint union.
\end{proof}

\subsection{}
We proceed to the construction of the invariant~$\rho_t
\colon \burn(K)\to\burn(k)$.
For every integer~$d\geq1$, we denote by~$R_d$
the finite $R$-algebra $R[u]/(u^d-t)$ and let $K_d$ be its field of fractions.
The extension $K_d$ of~$K$ is a $\mu_d$-torsor.

Let $X$ be a smooth proper $K$-variety.

By the semistable reduction theorem of 
\cite[chapter~II]{KempfKnudsenMumfordEtAl-1973},
there exists an integer~$d$ such that 
$X\otimes_K K_d$ has a 
\emph{semistable} proper model: it is regular, flat,
and its special fiber is 
a reduced divisor with strict normal crossings.

\begin{lemm}
For any multiple~$e$ of~$d$, one has $\rho(X\otimes_K K_e)=\rho(X\otimes_K K_d)$.
\end{lemm}
\begin{proof}
Replacing~$R$ by~$R_d$, we assume $d=1$.
Consider a semistable proper model~$\mathscr X$ of~$X$ over~$R$.
By \cite{KempfKnudsenMumfordEtAl-1973}, lemma~2, p.~103
(and the footnote on that page),
the base change  $\mathscr X\otimes_R R_e$ is normal, 
and its special fiber identifies with that of~$\mathscr X$.
(It is here that we use that we start from a \emph{semistable} model.)
The model $\mathscr X\otimes_R R_e$ is not regular in general, 
but it is a toroidal embedding without intersection of its generic fiber.
Consequently, the preceding lemma applies and we get the desired equality
$ \rho(X_e) =\rho(X)$.
\end{proof}
Thus one can define $\rho_t(X)$ as the class
$\rho(X\otimes_K K_d)$, for any integer~$d$ such that 
$X\otimes_K K_d$ admits a semistable model over~$R_d$.
Then, the arguments at the end of
the proof of~\cite[theorem~18]{KontsevichTschinkel-2019}
prove that the additive extension $\rho_t \colon \burn (K)\to\burn(k)$
is multiplicative.

\begin{rema}
For $\alpha\in\mathscr A$, let $e_\alpha$ be the multiplicity
of~$\Delta_\alpha$ in~$\Delta$;
for any nonempty subset~$A\subset\mathscr A$, let~$e_A$ 
be the greatest common divisor of
the~$e_\alpha$, for $\alpha\in A$.

As explained in \cite{KempfKnudsenMumfordEtAl-1973},
see also \cite[\S2.3]{Nicaise-2013},
especially proposition~2.3.2,
one can compute the normalization of~$\mathscr X\otimes R_d$
in terms of the given model~$\mathscr X$.
This gives an explicit decomposition of~$\rho_t(X)$
as a sum
\[ \sum_{\emptyset \neq A\subset \mathscr A_o}
 (-1)^{\Card(A)-1} [D'_A] \cdot \Lef ^{\Card(A)-1}, \]
where $\nu_A\colon D'_A\to D_A$ is some $\mu_{e_A}$-torsor defined as follows.

We identify the normal bundle of~$\Delta_A$ in~$\mathscr X$
as a direct sum of line bundles:
\[ \mathscr N_{\Delta_A}(\mathscr X) \simeq \bigoplus_{\alpha\in A}
  \mathscr N_{\Delta_\alpha}(\mathscr X)|_{\Delta_A} . \]
Let us consider its open subscheme~$\mathscr N_{\Delta_A}^\circ(\mathscr X)$
obtained by restricting to~$\Delta_A^\circ$
and taking out all “coordinate” hyperplanes.
This furnishes a morphism
\[ \nu_A \colon \mathscr N_{\Delta_A}^\circ(\mathscr X) \to
 \bigotimes_{\alpha\in A} \mathscr N_{\Delta_\alpha}(\mathscr X)^{\otimes e_\alpha}|_{\Delta_A^\circ} . \]
Since the uniformizer~$t$ has divisor 
$-\sum_{\alpha\in\mathscr A} e_\alpha \Delta_\alpha$ on~$\mathscr X$,
it trivializes the line bundle on the target of~$\nu_A$.
We set $\Delta'_A=\nu_A^{-1}(t)$.
By construction, the projection $\Delta'_A\to \Delta_A$ 
is a torsor with group~$\mu_{e_A}$.

However, contrary to the expectation of~\cite{KontsevichTschinkel-2019},
it is not clear how this construction lifts to a $\burn^{\mu}(k)$-valued
invariant. Note that the equivariant Burnside groups defined
in~\cite{KreschTschinkel-2022} have additional, nontrivial, relations.
\end{rema}

\subsection{}
Let us now explain how to define analogous specialization
homomorphisms in the context
of Burnside groups with volume forms.

For simplicity, we only consider
the case where $K$ has transcendence degree~$1$ over~$k$,
in which case the idea can be explained geometrically as follows.
We assume that there exists a smooth integral curve~$C$
together with a $k$-point $o\in C(k)$
such that $K=k(C)$ and $R=\mathscr O_{C,o}$.
We fix a local parameter $t\in R$ such that $V(t)=o$.

Let us consider a pair $(X,\omega)$ consisting of an integral
proper $K$-variety~$X$ of dimension~$n$
and a logarithmic $n$-form~$\omega$ on~$X$.

\subsection{}\label{ss.regularmodel}
Consider a regular flat proper model~$\mathscr X$ of~$X$ over~$C$,
let $\Delta=(\mathscr X_o)_\red$ be its reduced special fiber, and 
consider a divisor~$\mathscr D$ with relative normal crossings on~$\mathscr X$.
We assume that the divisor $\Delta+ \mathscr D$ has normal crossings.
Let $(\Delta_\alpha)_{\alpha\in\mathscr A}$ be the
family of irreducible components of~$\Delta$;
for $A\subset\mathscr A$, set $\Delta_A=\bigcap_{\alpha\in A}\Delta_\alpha$.
Similarly, let $(\mathscr D_\beta)_{\beta\in\mathscr B}$
be the family of irreducible components of~$\mathscr D$,
and for $B\subset\mathscr B$, let 
$\mathscr D_B=\bigcap_{\beta\in \mathscr B}\mathscr D_\beta$.

In this situation, 
\citet[\S3.3.2]{Deligne-1970} says that 
a meromorphic relative differential $m$-form on~$\mathscr X/C$
is logarithmic with respect to~$\Delta+\mathscr D$ 
if it is (locally) 
the image of a logarithmic $m$-form~$\tilde\omega$
in~$\Omega^m_{\mathscr X/k}$ with poles contained in $\Delta+\mathscr D$
under the  natural
morphism $\Omega^m_{\mathscr X/k}\to\Omega^m_{\mathscr X/C}$.

Consider a logarithmic relative $n$-form~$\omega$ on~$\mathscr X/C$.
We consider an associated volume form~$\omega'$ on~$\mathscr X$, 
defined locally by
\[ \omega' = \tilde \omega \wedge \mathrm dt/t, \]
where $\tilde\omega$ is any local lift of~$\omega$.
This form $\omega'$ is logarithmic
and we can compute its “residue along~$\Delta$” as
in~\S\ref{sec.residues}, only taking into account
the strata of the polar divisor of~$\omega'$
which are contained in the special fiber~$\Delta$.

There exists a subset $\mathscr A_o$ of~$\mathscr A$
and a subset~$\mathscr B_o$ of~$\mathscr B$
such that 
the polar divisor of~$\omega'$ is given by 
\[ \sum_{\alpha\in\mathscr A_o}  \Delta_\alpha + \sum_{\beta\in\mathscr B_o} \mathscr D_\beta. \]
We thus set
\[ \rho_t(\mathscr X,\omega)
 = \sum_{\substack{\emptyset \neq A \subset \mathscr A_o \\ B \subset\mathscr B_o} }(-1)^{\Card(A)+\Card(B)-1}
\rho_{\Delta_A\cap \mathscr D_B}(\mathscr X,\omega). \]
This is an element of $\Burn_n(\mathscr X_o/k)$.
\begin{prop}
Let $\mathscr Y$ be an irreducible smooth closed subscheme of~$\mathscr X$
which is transverse to $\Delta + \mathscr D$
and let $g\colon \mathscr X'\to\mathscr X$ 
be the blowing-up of~$\mathscr X$ along~$\mathscr Y$.
The form $g^*\omega$ on~$\mathscr X'$ is logarithmic and 
we have
\[ g_* \rho_t (\mathscr X',g^*\omega) = \rho_t (\mathscr X,\omega) \]
in $\Burn_n(\mathscr X_o/k)$.
\end{prop}
\begin{proof}
With the notation of~\S\ref{sec.residues}, the difference
\[ \rho (\mathscr X,\tilde\omega)  - \rho_t(\mathscr X,\omega) \]
is exactly the  part of $\rho(\mathscr X,\tilde\omega)$ 
which lies over the complement of the special fiber~$\mathscr X_o$
in~$\mathscr X$.
We have seen in theorem~\ref{theo.residue-bir} that
\[ g_* \rho(\mathscr X',\tilde\omega') = \rho(\mathscr X,\omega), \]
and a similar formula holds over $\mathscr X\setminus\mathscr X_o$.
This implies the proposition.
\end{proof}

\subsection{}
Starting from a smooth proper $K$-variety~$X$
and a logarithmic volume form~$\omega$ on~$X$,
we can define a model~$\mathscr X/C$, with $\Delta$ and~$\mathscr D$ 
as above, but the form~$\omega$
will not necessarily extend to a logarithmic relative form with respect
to $\Delta + \mathscr D$,
nor does 
the volume form~$\tilde\omega$ on~$\mathscr X$.
However, this can be achieved by multiplying~$\omega$
by a suitable power of the uniformizing element.

Similarly to~\S\ref{ss.regularmodel},
let us write the divisor of~$\tilde\omega$ on~$\mathscr X$ as
\[ \div_{\mathscr X}(\tilde\omega) 
= \sum_{\alpha\in\mathscr A} d_\alpha \Delta_\alpha
 + \sum_{\beta\in\mathscr B} d_\beta \mathscr D_\beta. \]
With this notation, the condition for~$\tilde\omega$ to be logarithmic
on~$\mathscr X$ is just that 
\[ d_\alpha \geq -1, \qquad d_\beta\geq -1. \]
In particular, while the conditions at the horizontal components follow 
from their counterparts on the generic fiber, 
those for the vertical components are not automatic.
On the other hand, for any $\kappa\in\Z$, the form $t^\kappa \tilde\omega$
is logarithmic if and only if 
\[ \kappa e_\alpha + d_\alpha \geq -1 \]
for all $\alpha\in\mathscr A$, that is, if and only if
$ \kappa \geq \kappa(\omega)$, 
where $\kappa(\omega)$ is defined by 
\[ \kappa(\omega) = - \inf_{\alpha\in\mathscr A} \frac{1+d_\alpha}{e_\alpha}. \]
Since the rational number
$\kappa(\omega)$ is defined in terms of logarithmic forms,
it only depends on the class of $(X,\omega)$ in~$\Burn_n(K)$,
and not on the actual model which is chosen to compute it.

\subsection{}
We assume for the moment that $\kappa(\omega) \in\Z$. 
This holds in particular if the special fiber~$\mathscr X_o$ is reduced.
Let then $\mathscr A_o$ be the subset of~$\mathscr A$
consisting of all $\alpha$ such that 
\[ \kappa (\omega) e_\alpha+d_\alpha = -1,\]
and let $\mathscr B_o$ be the subset of~$\mathscr B$
consisting of all $\beta$ such that $d_\beta=-1$.
The polar divisor of $t^\kappa\tilde\omega$ is equal to
\[ \sum_{\alpha\in\mathscr A_o} \Delta_\alpha
 + \sum_{\beta\in\mathscr B_o} \mathscr D_\beta, \]
and we set
\[ \rho_t(\mathscr X,\omega) = \rho_t(\mathscr X,t^{\kappa(\omega)}\omega) \]
in $\Burn_n(\mathscr X_o/k)$.

In the particular case where $\mathscr D$ is empty, 
the strata of the Clemens complex of the special fiber 
that actually appear in the definition of this class
are those defined by~\cite{KontsevichSoibelman-2006},
more precisely, by the adjustment provided by~\cite{MustataNicaise-2015a}.

\subsection{}
In the general case, the rational number~$\kappa(\omega)$
is not an integer.  
Let us consider the finite ramified extension~$K_d$ of~$K$, 
whose ramification index~$d$
is a multiple of the denominator of~$\kappa(\omega)$,
but which induces an isomorphism on the residue field.
Geometrically, this furnishes a morphism $\pi\colon C_d\to C$
which is ramified at the point~$o$, together
with a lift of~$o$ in~$C_d(k)$ (still denoted by~$o$),
and a distinguished uniformizing element~$t^{1/d}$.

We consider the extension of $(X,\omega)$ to~$K_d$ and
introduce a model $(\mathscr X_d,\omega_d)$ as above, over~$C_d$. 
Now, the corresponding $\kappa$-parameter is integral, so that 
any choice of a uniformizing element~$t^{1/d}$ in~$R_d$
induces a class 
$\rho_{t^{1/d}}(\mathscr X_d,\omega_d)$
in~$\Burn(k)$.
Combining these classes, we obtain the desired
group homomorphism
\[ \boldsymbol\rho_t \colon \Burn(K)  \to \Burn(k). \]
In this context, the variant of theorem~18 of~\citep{KontsevichTschinkel-2019}
is the following theorem.
\begin{theo}
The morphism~$\boldsymbol\rho_t\colon\Burn(K)\to\Burn(k)$ is a ring  homomorphism.
\end{theo}
\begin{proof}
The statement of theorem~18 in
\citep{KontsevichTschinkel-2019} is incorrect,
because it involves the $\burn^\mu(k)$-valued “invariant” on~$\burn(K)$
which is not well defined.
However, as we explained in the beginning of this section, 
this invariant becomes well defined after forgetting the $\mu$-torsor component,
and becomes the map denoted by~$\rho_t$ above.
The rest of their proof consists in explicitly
resolving the toric singularities
of a fiber product of two snc models and checking the 
required cancellations in $\burn(k)$.
Keeping track of the various logarithmic volume forms on the strata,
the same argument proves our theorem.
\end{proof}

\begin{rema}
In the case  of specialization of rationality, it has proved
fruitful to consider models with singularities on the special fiber,
mild enough so that 
the special fiber computes the specialization of
the birational type of the generic fiber.
This is in particular the case for rational double points.
A parallel study can be developped in the context
of varieties with logarithmic forms.
\end{rema}

\bibliography{burnvol}

\end{document}